\documentclass[10pt]{article}

\usepackage{geometry}
\RequirePackage[OT1]{fontenc}
\RequirePackage{amsthm,amsmath,amssymb}
\RequirePackage[colorlinks,citecolor=blue,urlcolor=blue]{hyperref}
\usepackage{stmaryrd}
\usepackage[sans]{dsfont}
\usepackage{authblk}

\usepackage[english]{babel}
\usepackage{latexsym}
\usepackage{bbm}
\usepackage[mathcal]{euscript}
\usepackage{graphicx}
\usepackage{pstricks}
\usepackage{pstricks-add}
\usepackage{pst-plot}
\usepackage{enumerate}
\usepackage{multicol}
\usepackage{subfigure}
\usepackage{mathrsfs}
\newcommand{\ind}{\mathbbm{1}} 
\newcommand{\II}{\mathbb{I}}   

\newcommand{\RR}{\mathbb{R}}    

\numberwithin{equation}{section}
\theoremstyle{plain}
\newtheorem{theorem}{Theorem}[section]

\newtheorem{proposition}[theorem]{Proposition}

\newtheorem*{example}{Example}
\theoremstyle{definition}
\newtheorem{definition}{Definition}[section]
 
\begin{document}

\title{Inverse $M$-matrix, a new characterization}

\author{Claude Dellacherie\thanks{Laboratoire Rapha\"el Salem,
UMR 6085, Universit\'e de Rouen, Site du Madrillet, 76801 Saint \'Etienne 
du Rouvray Cedex, France. email: Claude.Dellacherie@univ-rouen.fr} ,
Servet Mart\'inez\thanks{CMM-DIM;  Universidad de Chile; UMI-CNRS 2807; 
Casilla 170-3 Correo 3 Santiago; Chile. email: smartine@dim.uchile.cl} ,
Jaime San Mart\'in\thanks{CMM-DIM;  Universidad de Chile; UMI-CNRS 2807; 
Casilla 170-3 Correo 3 Santiago; Chile. email: jsanmart@dim.uchile.cl}}


\maketitle

\begin{abstract} 
In this article we present a new characterization of inverse $M$-matrices, inverse row diagonally
dominant $M$-matrices and inverse row and column diagonally dominant $M$-matrices, 
based on the positivity of certain inner products.
\end{abstract}

\medskip


\setlength{\parskip}{0.1cm}

{
\bfseries
{\small
\hypersetup{linkcolor=black}
\tableofcontents
}
\normalfont
}

\section{Introduction and Main Results.}

In this short note, we give a new characterization of inverses $M$-matrices, inverses of row diagonally
dominant $M$-matrices and inverses of row and column diagonally dominant $M$-matrices. 
This is done in terms of a certain inner product to be nonnegative (see \eqref{cond:5}, \eqref{cond:6} 
and \eqref{cond:4}, respectively). These characterizations are stable under limits, that is, 
if an operator $\mathscr{U}$ can be approximated by a sequence of matrices $(U_k)_k$ 
in such a way that also the corresponding inner products converge (for example in $L^2$)  then
the limit operator will satisfy the same type of inequality. This is critical to show, for example, 
that $\mathscr{U}$ is the $0$-potential of a Markov resolvent or a Markov semigroup
because as we will see this inequality implies a strong principle called Complete Maximum Principle 
in Potential Theory (see for example \cite{MarcusRosen2006}, chapter 4). In the matrix case, this
corresponds to the inverse of a row diagonally dominant $M$-matrix (see Theorem \ref{the:3} below).

We continue with the formal definition of a potential matrix.
\begin{definition}
\label{def:1} 
A nonnegative, nonsingular matrix $U$ is called a {\it potential} 
if its inverse $M=U^{-1}$ is a row diagonally
dominant $M$-matrix, that is,
\begin{eqnarray}
&& \forall i\neq j  \hspace{0.3cm}M_{ij}\le 0;\label{cond:1}\\
&& \forall i   \hspace{0.9cm} M_{ii} > 0;\label{cond:2}\\
&& \forall i   \hspace{0.9cm}  |M_{ii}|\ge \sum\limits_{j\neq i} |M_{ij}|. \label{cond:3}
\end{eqnarray}

Also, $U$ is called a {\it double potential} if $U$ and $U^t$ are potentials.
\end{definition}

We point out that conditions \eqref{cond:1} and \eqref{cond:3} imply condition \eqref{cond:2}.
Indeed, notice that these two conditions imply that $|M_{ii}|\ge 0$, but if $M_{ii}=0$, then 
for all $j$ we would have $M_{ij}=0$, which is not possible because we assume that $M$ is nonsingular.
Finally, $M_{ii}$ cannot be negative, otherwise $1=\sum_j M_{ij} U_{ji}\le 0$, which is not possible.
We also notice that if $U$ is a symmetric potential, then clearly it is a double potential.

In what follows for a vector $x$, we denote by $x^+$ its positive part, 
which is given by $(x^+)_i=(x_i)^+$. Similarly, $x^-$ denotes the negative part of $x$. 
Also, we denote by $\langle \, ,\, \rangle$, the standard
euclidean inner product, and $\ind$ is the vector whose entries are all ones.
We are in a position to state our main result.

\begin{theorem} 
\label{the:1}
Assume that $U$ is a nonsingular nonnegative matrix of size $n$. 
\begin{enumerate}[(i)]
\item If $~U$ satisfies the following inequality: for all $x\in \RR^n$
\begin{equation}
\label{cond:4} 
\langle (Ux-\ind)^+, x\rangle \ge 0
\end{equation}
then $U$ is a potential.

\item Reciprocally, if $~U$ is a double potential then it satisfies \eqref{cond:4}. 
\end{enumerate}
In particular, if $U$ is a nonnegative nonsingular symmetric matrix, 
then $U$ is a potential iff it satisfies \eqref{cond:4}.
\end{theorem}
\begin{example} Here is an example of a potential matrix $U$, for which (\ref{cond:4}) 
does not hold. Consider
$$
U=\begin{pmatrix} 2 & 100\\ 1 & 100\end{pmatrix},
$$
whose inverse is $M=U^{-1}$
$$
 M=\begin{pmatrix} 1& -1\\ -1/100 & 1/50\end{pmatrix}=\II-\begin{pmatrix} 0 & 1\\ 1/100 & 49/50\end{pmatrix}
$$
a row diagonally dominant $M$-matrix.
Nevertheless,  
$$
\langle (Uv-\ind)^+, v \rangle=(2x+100y-1)^+x+(x+100y-1)^+y=-5.3,
$$ 
for $x=-0.5, y=0.2$. Notice that $U^t$ is not a potential because its inverse, although it is an $M$-matrix, it fails
to be row diagonally dominant.
\end{example}

To generalize Theorem \ref{the:1} to include all inverse $M$-matrices we consider the following
two diagonal matrices $D,E$. Here, we assume that $U$ is a nonnegative nonsingular 
matrix, or more general, it is enough to assume that $U$ is nonnegative and it 
has at least one positive element per row and column. Let us define
$D$ as the diagonal matrix given by, for all $i$
$$
D_{ii}=\left(\sum_j U_{ij}\right)^{-1},
$$
as the reciprocal of the $i$-th row sum. Similarly, consider $E$ the diagonal matrix given by, for all
$i$
$$
E_{ii}=\left(\sum_j U_{ji}\right)^{-1},
$$
the reciprocal of the $i$-th column sum. We point out that matrices $D,E$ are computed directly from
$U$.
\begin{theorem} 
\label{the:2}
Assume that $U$ is a nonsingular nonnegative matrix of size $n$. 
\begin{enumerate}[(i)]
\item $U$ is an inverse $M$-matrix iff $D\,UE$ is a double potencial, which is further equivalent
to the following inequality: for all $x\in \RR^n$
\begin{equation}
\label{cond:5} 
\langle (Ux-D^{-1}\ind)^+,D E^{-1} x\rangle \ge 0.
\end{equation}
\item $U$ is a potential iff $~UE$ is a double potential, which is equivalent to the inequality:
for all $x\in \RR^n$
\begin{equation}
\label{cond:6} 
\langle (Ux-\ind)^+,E^{-1} x\rangle \ge 0.
\end{equation}
\end{enumerate}
\end{theorem}
\begin{proof} $(i)$ Assume that $M=U^{-1}$ is an $M$-matrix. Then, $W=D\,UE$ is a double potential.
Indeed, it is clear that $N=W^{-1}$ is an $M$-matrix.
Now, consider $\mu=E^{-1}\ind$, then
$$
W\mu=DU\ind=\ind,
$$
by the definition of $D$. This means that $N\ind=W^{-1}\ind=\mu\ge 0$, and so
$N$ is a row diagonally dominant matrix. Similarly, if we take $\nu=D^{-1}\ind$, we have
$$
\nu^t W=\ind^t UE=\ind^t.
$$
This proves that $\ind^t N=\nu^t\ge 0$ and therefore, we conclude that 
$N$ is a column diagonally dominant matrix. In summary,
$W$ is a double potential. Conversely, if $W$ is a double potential, in particular it is an inverse $M$-matrix,
which implies that $U$ is an inverse $M$-matrix. 

Let us prove that $U$ being an inverse $M$-matrix is equivalent to \eqref{cond:5}. 
We first assume $W$ is a double potential, then from Theorem \ref{the:1}, 
we have for all $x\in \RR^n$
$$
0\le \langle (D\,UE\,x-\ind)^+,x\rangle=\langle (DUy-\ind)^+,E^{-1}y\rangle
=\langle (Uy-D^{-1}\ind)^+,DE^{-1}y\rangle,
$$
which is condition \eqref{cond:5}. Here, we have used the straightforward to prove property that
for a diagonal matrix, with positive diagonal elements, it holds $(Dz)^+=D z^+$.

Conversely, assume that $U$ satisfies \eqref{cond:5} then, we obtain that
$W$ satisfies \eqref{cond:4} in Theorem \eqref{the:1} and therefore it is an inverse $M$-matrix. 
So, $U$ is an inverse $M$-matrix, proving the desired equivalence.

\noindent$(ii)$ This time we take $W=UE$. Since $U$ is a potential, there exists
a nonnegative vector $\mu$, such that $U\mu=\ind$, then $UEE^{-1}\mu=\ind$ and $W$
is a potential. On the other hand, $\ind^t UE=\ind^t$, and therefore $W$ is a double
potential. The rest follows similarly as in the proof of $(i)$.
\end{proof}

The next theorem is a complement to Theorem 1.1. One way to approach this result is 
by making the change of variables $y=Ux$ in \eqref{cond:4}. 

\begin{theorem}
\label{the:1M}
Assume $M$ is a matrix of size $n$. Then

\begin{enumerate}[(i)]

\item If $M$ satisfies the inequality, for all $x\in \RR^n$
\begin{equation}
\label{cond:M} 
\langle (x-\ind)^+, Mx\rangle \ge 0,
\end{equation}
then $M$ satisfies the the following structural properties
\begin{eqnarray}
&& \forall i\neq j  \hspace{0.3cm}M_{ij}\le 0;\label{cond:1'}\\
&& \forall i   \hspace{0.9cm} M_{ii} \ge 0;\label{cond:2'}\\
&& \forall i   \hspace{0.9cm}  M_{ii}\ge \sum\limits_{j\neq i} -M_{ij} \label{cond:3'}.
\end{eqnarray}
That is, $M$ is a $Z$-matrix, with nonnegative diagonal elements and it is a row diagonally dominant matrix.

\item If $M$ is a $Z$-matrix, with nonnegative diagonal elements and it is a row and column diagonally dominant matrix, 
then it satisfies \eqref{cond:M}.
\end{enumerate}

\end{theorem}

There is a vast literature on $M$-matrices and inverse $M$-matrices, the interested reader
may consult the books by Horn and Johnson \cite{Horn1985} and \cite{Horn1991}, among others.
In particular for the inverse $M$-problem we refer to the pioneer work of \cite{Fiedler}, \cite{Fiedler1983}
and \cite{Willoughby}.
Some results in the topic can be seen in 
\cite{Brandts2016}, \cite{Carmona}, \cite{dell2009}, \cite{dell2011}, \cite{Fiedler1998}, \cite{Fiedler2000},
\cite{Hogben}, \cite{Johnson}, \cite{Johnson2005}, \cite{Koltracht}, 
and \cite{Neumann}. The special relation of this problem
to ultrametric matrices in  \cite{dell1996}, \cite{Martinez1994}, \cite{McDonald1995}, \cite{McDonald1998}, 
\cite{Nabben2001}, \cite{Nabben1994}, \cite{Nabben1995}, \cite{Nabben1998}. 
Finally, for the relation between $M$-matrices and inverse $M$-matrices with Potential Theory see
our book \cite{libroDMSM2014}.

\section{proof of Theorem \ref{the:1} and Theorem \ref{the:1M}.}

The proof of Theorem \ref{the:1} is based on what is called the 
{\it Complete Maximum Principle} (CMP), which
we recall for the sake of completeness.

\begin{definition} 
A nonnegative matrix $U$ of size $n$, is said to satisfies the CMP if for all $x\in \RR^n$
it holds: 
$$
\sup\limits_{i} (Ux)_i \le \sup\limits_{i: x_i\ge 0} (Ux)_i,
$$
where by convention $\sup\limits_{\varnothing}=0$.
\end{definition}
The CMP says that if $x$ has at least one nonnegative coordinate then the maximum value
among the coordinates of $y=Ux$ is attained at some coordinate $i$ such that $x_i$ is nonnegative. 
An alternative equivalent definition, which is the standard in Potential Theory reads as follows, 
$U$ is a potential if for all $x$ it holds: whenever  $(Ux)_i\le 1$ on the coordinates where $x_i\ge 0$, 
then $Ux\le \ind$. The importance of this principle is given by the next result.

\begin{theorem} 
\label{the:3}
Assume $U$ is a nonnegative matrix. 
\begin{enumerate}[(i)]
\item If $~U$ is nonsingular, then $U$ satisfies the CMP iff $~U$ is a potential, that is,
$M=U^{-1}$ is a row diagonally dominant $M$-matrix.
\smallskip
\item $U$ satisfies the CMP then for all $a\ge 0$ the matrix $U(a)=a\II +U$ satisfies
the CMP and for all $a>0$ the matrix $U(a)$ is nonsingular.
\end{enumerate}
\end{theorem} 
The proof of $(i)$ in this theorem goes back to Choquet and Deny \cite{choquet1956} (Theorem 6, page 89). 
For a generalization of this result and a more matrix flavor of it, see Theorem 2.9 in \cite{libroDMSM2014}.

Assume that $U$ is a nonnegative matrix and satisfies the CMP, if the diagonal of
$U$ is strictly positive, which happens when $U$ is nonsingular, then there exists an equilibrium 
potential, that is, a nonnegative vector $\mu$ solution of the problem 
$$
U\mu=\ind,
$$
see for example $(v)$ Lemma 2.7 in \cite{libroDMSM2014}.

This vector $\mu$ plays an important role and it is related to the fact that $U^{-1}$ is row
diagonally dominant, when $U$ is nonsingular. In fact, in this case $\mu=U^{-1}\ind\ge 0$.

Now, we are in a position to prove Theorem \ref{the:1}.

\begin{proof}(Theorem \ref{the:1})

\noindent$(i)$ We shall prove that $U$ satisfies the CMP. For that purpose consider
$x\in \RR^n$, which has at least one nonnegative coordinate. If $(Ux)_i\le 1$, for those 
coordinates $i$ such that $x_i\ge 0$, then from condition \eqref{cond:4} we conclude that
$$
0\le \langle (Ux-\ind)^+,x\rangle=\langle (Ux-\ind)^+,x^-\rangle,
$$
which implies that $((Ux-\ind)^+)_i=0$ if $x_i<0$, proving that $U$ satisfies the CMP.
Hence, from Theorem \ref{the:3} we have that $M=U^{-1}$ is a row diagonally dominant
$M$-matrix.

\noindent$(ii)$ Assume that $U, U^t$ are potential matrices of size $n$. Then $M=U^{-1}$
is a column and row diagonally $M$-matrix, which is equivalent to have
$M=k(\II-P)$, for some  constant $k>0$ and a double substochastic matrix $P$, that is,
$P$ is a nonnegative matrix and for all $i$ it holds $\sum_j P_{ij}\le 1, \sum_j P_{ji}\le 1$.
 
We define $\mu=M\ind\ge 0$ and  $\xi=U(x-\mu)=Ux-\ind$ to get
$$
\begin{array}{l}
\langle (Ux-\ind)^+, x \rangle=\langle (Ux-U\mu)^+, x \rangle=\langle \xi^+, M\xi+\mu \rangle
=\langle \xi^+,k \xi +\mu \rangle-k \langle \xi^+,P\xi \rangle\\
\\
=k\Big(\langle \xi^+, \xi^+ \rangle-\langle \xi^+,P\xi \rangle\Big)+\langle \xi^+,\mu \rangle.
\end{array}
$$
Since $P\ge 0$, we get 
$$
\langle \xi^+,P\xi \rangle \le \langle \xi^+,P\xi^+ \rangle= \langle \xi^+,\frac12(P+P^t)(\xi)^+ \rangle
\le \langle \xi^+, \xi^+ \rangle.
$$
The last inequality holds because the nonnegative symmetric matrix $\frac12(P+P^t)$ 
is sub-stochastic and therefore its spectral radius is smaller than 1, which
implies that for all $z\in \RR^n$ it holds $ \langle z,\frac12(P+P^t)z \rangle\le \langle z, z \rangle$. We get
the inequality
$$
\langle (Ux-\ind)^+, x \rangle\ge \langle (Ux-\ind)^+,\mu \rangle\ge 0,
$$ 
which shows the result.
\end{proof}

\begin{proof}(Theorem \ref{the:1M}) The reader may consult \cite{Horn1991}
Theorem 2.5.3, for some properties about $M$-matrix that are needed in this proof.

$(i)$. Assume that $M$ is a matrix, of size $n$, that satisfies \eqref{cond:M}. 
In order to prove that condition \eqref{cond:1'} holds fix $i\in \{1,\cdots,n\}$ and consider
a vector $x$ such that $x_i>1$ and $x_k=0, k\neq i$. Then \eqref{cond:M} implies
$$
0\le \langle (x-\ind)^+, Mx\rangle=(x_i-1) M_{ii} x_i,
$$
from where we deduce $M_{ii}\ge 0$, proving that \eqref{cond:1'} holds.

To prove \eqref{cond:2'} consider $i\neq j$ fixed and take a vector $x$ such that $x_i>1,x_j<0$ and 
$x_k=0, k\neq i,j$. Then
$$
0\le \langle (x-\ind)^+, Mx\rangle=(x_i-1) (M_{ii} x_i+M_{ij}x_j)
$$
By taking $x_j$ a large negative number, we conclude that this inequality can hold only if
$M_{ij}\le 0$, proving \eqref{cond:2'}. 

Now, we prove condition \eqref{cond:3'}. For that purpose we consider $i$ fixed and we take
$x\in \RR^n$ such that $x_i>1$ and $x_j=1$ for all $j\neq i$. Then
$$
0\le \langle (x-\ind)^+, Mx\rangle=(x_i-1)\Big(x_iM_{ii}+ \sum_{j\neq i} M_{ij}\Big).
$$
This implies that $x_iM_{ii}+ \sum\limits_{j\neq i} M_{ij}\ge 0$ holds for all $x_i\ge 1$ and therefore
$M_{ii}+ \sum\limits_{j\neq i} M_{ij}\ge 0$, proving that $M$ satisfies \eqref{cond:3'}.

Part $(ii)$ follows from Theorem 1.1 by considering a perturbation of $M$. For $\theta>0$ take 
$M(\theta)=\theta\II+M$. By hypothesis $M(\theta)$ is a strictly row and column dominant $Z$-matrix, proving
that $M(\theta)$ is an $M$-matrix. Its inverse, $U(\theta)=(M(\theta))^{-1}$, is a double potential and 
therefore it satisfies inequality \eqref{cond:4}. 
Take $y\in \RR^n$ and consider $x=x(\theta)=M(\theta)y$ to obtain
$$
0\le \langle (U(\theta)x-\ind)^+, x\rangle=
\langle (y-\ind)^+, My\rangle+\theta \langle (y-\ind)^+, y\rangle.
$$
The result follows by taking $\theta\downarrow 0$.
\end{proof}

\section{Some complements}
In Potential Theory, particularly when dealing with infinite dimensional spaces, most of the time
a potential $U$ is singular. According to Theorem \ref{the:3}, in case $U$ is nonsingular, our 
definition of a potential matrix (see Definition \ref{def:1}) and the CMP are equivalent. The latter
makes sense even for singular matrices and this should be the right definition for a matrix to be
a potential. Notice that in $(ii)$ Theorem \ref{the:3} says that a potential in this sense, is the limit
of nonsingular potencial matrices, which also holds in infinite dimensional spaces.

Theorem \ref{the:1} can be extended to include singular potential matrices as follows.

\begin{theorem} Assume $U$ is a nonnegative matrix.
\begin{enumerate}[(i)]
\item If $~U$ satisfies \eqref{cond:4} then $U$ satisfies the CMP.
\smallskip

\item Reciprocally, if $~U, U^t$ satisfy the CMP, then $U$ (and $U^t$) satisfies \eqref{cond:4}. 
\end{enumerate}
That is, for a symmetric matrix $U$, condition \eqref{cond:4} and CMP are equivalent.
\end{theorem}

\begin{proof} The proof of $(i)$ is identical to the one of Theorem \ref{the:1} $(i)$.

\noindent$(ii)$. Consider as in Theorem \ref{the:3} a perturbation of $U$, given by
$U(a)=a\II+U$, for $a>0$. Since $U(a),U^t(a)$ are nonsingular potential matrices
we can use Theorem \ref{the:1} to conclude that for all $a>0$ and $x\in \RR^n$
one has
$$
0\le \langle (U(a)x-\ind)^+, x\rangle.
$$
Now, it is enough to take the limit as $a\downarrow 0$, to conclude the result.
\end{proof}

The question now is: is there a principle like CMP, that characterizes inverses $M$-matrices?
The answer is yes, and it is given by the following principle taken from Potential Theory.

\begin{definition}
A nonnegative matrix $U$ is said to satisfy the {\it the domination principle} (DP) if for
any nonnegative vectors $x,y$ it holds $(Ux)_i\le (Uy)_i $ for those coordinates $i$
such that $x_i > 0$, 
then $Ux\le Uy$.
\end{definition}
Theorem 2.15 in \cite{libroDMSM2014} is exactly this characterization, which we copy here
for the sake of completeness.

\begin{theorem} Assume that $U$ is a nonnegative nonsingular matrix. Then,
$U^{-1}$ is an $M$-matrix iff $~U$ satisfies DP.

\end{theorem}

It is interesting to know a relation between CMP and DP, which is given by Lemma 2.13 in 
\cite{libroDMSM2014}.

\begin{proposition} Assume that $U$ is a nonnegative matrix with positive diagonal
elements. Then, the following are equivalent
\begin{enumerate}[(i)]
\item $U$ satisfies the CMP
\item $U$ satisfies the DP and there exists a nonnegative vector
$\mu$ solution to $U\mu=\ind$.
\end{enumerate}
\end{proposition}

Finally let us recall a simple algorithm to check when a nonnegative matrix that satisfies the
CMP or DP is nonsingular (see Corollary 2.46 and Corollary 2.47 in \cite{libroDMSM2014}).

\begin{proposition} Assume that $U$ is a nonnegative matrix, that satisfies either
CMP or DP, then the following are equivalent
\begin{enumerate}[(i)]
\item $U$ is nonsingular 
\item  not two columns of $~U$ are proportional.
\end{enumerate}
\end{proposition}
There is a lack of symmetry in this result from columns and rows, because CMP is not stable
under transposition. On the other hand DP is stable under transposition, so in this case
$U$ is nonsingular iff no two rows are proportional.

\section*{Acknowledgement}  
Authors S.M. and J.SM. where partially founded by CONICYT, project BASAL AFB170001. The authors 
thank an anonymous referee for comments and suggestions.



\begin{thebibliography}{99}
\bibitem{Brandts2016} Brandts, J., Cihangir, A., Geometric aspects of the symmetric inverse M-matrix problem. 
 Linear Algebra Appl. 506, 33--81, (2016). 

\bibitem{Carmona} Carmona, A., Encinas, A.M., Mitjana, M., On the M-matrix inverse problem 
for singular and symmetric Jacobi matrices. 
Linear Algebra Appl. 436(5), 1090--1098, (2012). 

\bibitem{choquet1956} Choquet, G., Deny, J., Mod\`eles finis en th\'eorie du potentiel.
 J. d'Analyse Math\'ematique 5, 77--135, (1956).

\bibitem{dell1996} Dellacherie, C., Mart\'{\i}nez, S., San Mart\'{\i}n,  J.,  
 Ultrametric matrices and induced Markov chains.  Adv. App. Math.  
17, 169--183, (1996).

\bibitem{dell2009} Dellacherie, C., Mart\'{\i}nez, S., San Mart\'{\i}n,  J.,
 Hadamard functions of inverse M-Matrices. 
SIAM J. Matrix Anal. Appl. 31(2), 289--315, (2009).
 
\bibitem{dell2011} Dellacherie, C., Mart\'{\i}nez, S., San Mart\'{\i}n,  J.,
 Hadamard functions that preserve inverse $M$-matrices. 
SIAM J. Matrix Anal. Appl. 33(2), 501--522, (2012).

\bibitem{libroDMSM2014}  Dellacherie, C., Mart\'{\i}nez, S., San Mart\'{\i}n,  J.,
 Inverse $M$-matrices and Ultrametric matrices. 
Lecture Notes in Mathematics 2118, Springer (2014).

\bibitem{Fiedler} Fiedler, M., Relations between the diagonal elements of an M-matrix and the inverse matrix. 
(Russian) Mat.-Fyz. Casopis Sloven. Akad. Vied 12, 123--128, (1962).

\bibitem{Fiedler1998} Fiedler, M.,  Some characterizations of symmetric inverse M -matrices. 
Proceedings of the Sixth Conference of the International Linear Algebra Society (Chemnitz, 1996).
 Linear Algebra Appl. 275/276, 179--187, (1998).

\bibitem{Fiedler2000} Fiedler, M.,  Special ultrametric matrices and graphs. 
SIAM J. Matrix Anal. Appl. 22, 106--113, (2000).

\bibitem{Fiedler1983} Fiedler, M., Schneider, H.,  Analytic functions of M-matrices and generalizations. 
Linear Multilinear Algebra 13, 185--201, (1983).

\bibitem{Hogben} Hogben, L., The symmetric M-matrix and symmetric inverse M-matrix completion problems. 
Linear Algebra Appl. 353, 159--168, (2002).
 
\bibitem{Horn1985} Horn, R., Johnson, C.R.,  Matrix Analysis. Cambridge University Press, Cambridge (1985).

\bibitem{Horn1991} Horn, R., Johnson, C.R.,  Topics in Matrix Analysis. Cambridge University Press, Cambridge,
(1991).

\bibitem{Johnson} Johnson, C.R., Smith, R.,
The completion problem for M-matrices and inverse M-matrices. (English summary) 
Proceedings of the Fourth Conference of the International Linear Algebra Society (Rotterdam, 1994). 
Linear Algebra Appl. 241/243, 655--667, (1996). 

\bibitem{Johnson2005} Johnson, C.R.; Olesky, D. D., Rectangular submatrices of inverse M-matrices and 
 the decomposition of a positive matrix as a sum. Linear Algebra Appl. 409, 87--99, (2005). 

\bibitem{Koltracht} Koltracht, I., Neumann, M., On the inverse M-matrix problem for real symmetric positive-definite 
Toeplitz matrices. SIAM J. Matrix Anal. Appl. 12, no. 2, 310--320, (1991). 

\bibitem{Lewin} Lewin, M., Neumann, M., On the inverse M-matrix problem for (0,1)-matrices. 
Linear Algebra Appl. 30, 41--50, (1980).

\bibitem{Martinez1994} Mart\'inez , S., Michon, G., San Mart\'in, J., 
Inverses of ultrametric matrices are of Stieltjes types.
SIAM J. Matrix Anal. Appl. 15, 98--106, (1994).

\bibitem{McDonald1995} McDonald, J.J., Neumann, M., Schneider, H., Tsatsomeros, M.J., 
Inverse M-matrix inequalities  and generalized ultrametric matrices. Proceedings of the Workshop 
``Nonnegative Matrices, Applications and Generalizations'' and the 
Eighth Haifa Matrix Theory Conference (Haifa, 1993). Linear Algebra Appl. 220, 321--341, (1995). 

\bibitem{McDonald1998} McDonald, J.J., Neumann, M., Schneider, H., Tsatsomeros, M.J., Inverse tridiagonal
Z-matrices. Linear Multilinear Algebra 45, 75--97, (1998).

\bibitem{MarcusRosen2006} Marcus, M., Rosen, J.,  Markov Processes, 
Gaussian Processes and Local times, Cambridge University Press (2006).

\bibitem{Nabben2001} Nabben, R., On Green's matrices of trees. SIAM J. Matrix Anal. Appl. 22(4), 1014--1026, (2001).

\bibitem{Nabben1994} Nabben, R.,  Varga, R.S., A linear algebra proof that the inverse of a strictly 
ultrametric matrix is a strictly diagonally dominant Stieltjes matrix. SIAM J. Matrix Anal. Appl. 15, 107--113, (1994).

\bibitem{Nabben1995} Nabben, R.,  Varga, R.S., Generalized ultrametric matrices -- a class of inverse M-matrices.
Linear Algebra Appl. 220, 365--390, (1995).

\bibitem{Nabben1998} Nabben, R.,  Varga, R.S., On classes of inverse Z-matrices. 
Special issue honoring Miroslav Fiedler and Vlastimil Pt\'ak.
Linear Algebra Appl. 223/224, 521--552, (1998).

\bibitem{Neumann} Neumann, M.; Sze, N., On the inverse mean first passage matrix 
problem and the inverse M-matrix problem. Linear Algebra Appl. 434(7), 1620--1630, (2011). 

\bibitem{Willoughby} Willoughby, R. A., The inverse M-matrix problem. 
Linear Algebra and Appl. 18(1), 75--94, (1977). 
\end{thebibliography}
\end{document}